\documentclass[12pt]{amsproc}
\usepackage{amssymb}

\newtheorem*{theorem}{Theorem}
\newtheorem{corollary}{Corollary}
\newtheorem{lemma}{Lemma}

\begin{document}

\author[V.~Lebedev]{by\\ \textsc{Vladimir Lebedev} (Moscow)}
\address{National Research University Higher School of Economics,
34~Tallinskaya Str., Moscow 123458, Russia}

\email{lebedevhome@gmail.com}

\title[Quantitative aspects of the Beurling--Helson
theorem]{Quantitative aspects of the Beurling--Helson theorem:
Phase functions of a special form}

\subjclass[2010]{Primary 42B35; Secondary 42B05, 42A20}

\keywords{Absolutely convergent Fourier series, superposition
operators, Beurling--Helson theorem.}

\date{2/NOVEMBER/2016}

\begin{abstract}
We consider the space $A(\mathbb{T}^d)$ of absolutely
convergent Fourier series on the torus~$\mathbb{T}^d$. The
norm on $A(\mathbb{T}^d)$ is naturally defined by
$\|f\|_{A}=\|\widehat{f}\|_{l^1}$, where $\widehat{f}$ is the
Fourier transform of a function $f$. For real
functions~$\varphi$ of a certain special form on $\mathbb
T^d, \,d\geq 2,$ we obtain lower bounds for the norms
$\|e^{i\lambda\varphi}\|_A$ as $\lambda\rightarrow\infty$. In
particular, we show that if $\varphi(x, y)=a(x)|y|$ for
$|y|\leq\pi$, where $a\in A(\mathbb{T})$ is an arbitrary
nonconstant real function, then
$\|e^{i\lambda\varphi}\|_{A(\mathbb{T}^2)}\gtrsim |\lambda|$.
\end{abstract}

\maketitle

\section{Introduction}\label{s1}

Let $A(\mathbb{T}^d)$ be the space of all continuous
functions~$f$ on the torus
$\mathbb{T}^d=\mathbb{R}^d/(2\pi\mathbb{Z})^d$ such that the
sequence $\widehat{f}=\{\widehat{f}(n), n\in\mathbb{Z}^d\}$ of
Fourier coefficients of $f$ is in $l^1(\mathbb{Z}^d)$. The
space~$A(\mathbb{T}^d)$ is a Banach space with respect to the
natural norm
$\|f\|_{A(\mathbb{T}^d)}=\|\widehat{f}\|_{l^1(\mathbb{Z}^d)}$.
It is well known that $A(\mathbb{T}^d)$ is a Banach algebra
(with pointwise multiplication of functions). Here, as usual,
$$
\widehat{f}(n)=\frac{1}{(2\pi)^d}\int_{\mathbb{T}^d}f(t)e^{-i(n,
t)}dt, \qquad n\in\mathbb{Z}^d,
$$
$\mathbb{R}$ and $\mathbb{Z}$ are the additive groups of reals
and integers, respectively, and $(\cdot, \cdot)$ is the inner
product.

Let $\varphi$ be a continuous mapping of the circle
$\mathbb{T}$ into itself, that is, a continuous function
$\varphi \colon \mathbb{R}\rightarrow\mathbb{R}$ satisfying the
condition $\varphi(t+2\pi)=\varphi(t) (\mathrm{mod}\,2\pi)$. By
the well-known Beurling--Helson theorem~\cite{1} (see
also~\cite[Sec.~VI. 9]{6} and \cite{7}), if
$\|e^{in\varphi}\|_{A(\mathbb{T})}=O(1)$, $n\in \mathbb{Z}$,
then $\varphi$ is affine. In other words, in this case we have
$\varphi(t)=\nu t+\varphi(0)$, where
$\nu\in\mathbb{Z}$.\footnote{This version of the
Beurling--Helson theorem is due to Kahane.} This theorem
readily gives the solution to the Levy problem on the
description of endomorphisms of the algebra $A(\mathbb{T})$.
All these endomorphisms are trivial; i.e., they have the form
$f(t)\rightarrow f(\nu t+t_0)$. This implies that only trivial
changes of variable are admissible in $A(\mathbb T)$. Another
version of the Beurling--Helson theorem concerns power-bounded
operators on $l^1$: if $U$ is a bounded translation invariant
operator from $l^1(\mathbb Z)$ to itself such that
$\|U^n\|_{l^1\rightarrow l^1}=O(1), ~n \in \mathbb Z$, then
$U=\xi S$, where $\xi$ is a complex number, $|\xi|=1$, and $S$
is a translation. Note also the following multidimensional
version of the Beurling--Helson theorem. (The multidimensional
case easily reduces to the one-dimensional one.) Let $\varphi
\colon \mathbb{T}^d\rightarrow\mathbb{T}$ be a continuous
mapping such that $\|e^{in\varphi}\|_{A(\mathbb{T}^d)}= O(1)$;
then $\varphi(t)=(\nu, t)+\varphi(0)$ (where
$\nu\in\mathbb{Z}^d$).

At the same time, although the Beurling--Helson theorem
establishes an unbounded growth of the norms
$\|e^{in\varphi}\|_{A}$ for non-affine mappings $\varphi$, the
character of growth of these norms is unclear.

We note, that instead of non-affine mappings~$\varphi : \mathbb
T^d\rightarrow\mathbb T$, one can consider real nonconstant
functions~$\varphi$ on~$\mathbb T^d$ and study the behavior of
the norms $\|e^{i\lambda\varphi}\|_{A(\mathbb T^d)}$ for large
frequencies $\lambda\in\mathbb R$ (without assuming the
frequences to be integer).\footnote{If $\varphi : \mathbb
T^d\rightarrow\mathbb T$ is continuous, then for some
$k\in\mathbb Z^d$ the function $\psi(t)=\varphi(t)+(k, t)$ is a
(real) continuous function on $\mathbb T^d$ and
$\|e^{in\varphi}\|_{A(\mathbb T^d)}=\|e^{in\psi}\|_{A(\mathbb
T^d)}$.}

It is easy to show that if $\varphi$ is a real $C^1$ -smooth
function on the circle $\mathbb T$, then
$\|e^{i\lambda\varphi}\|_{A(\mathbb T)}\lesssim|\lambda|^{1/2}$
as $\lambda\rightarrow\infty$ (see, e.g.,~\cite[Sec.~VI.
3]{6}). On the other hand lower bound for $\varphi\in
C^2(\mathbb T)$ has long been known; namely, if $\varphi\in
C^2(\mathbb T)$ is a real nonconstant function, then
$\|e^{i\lambda\varphi}\|_{A(\mathbb T)}\gtrsim|\lambda|^{1/2}$.
This estimate is contained implicitly in the work by
Leibenson~\cite{15} and in explicit form was obtained by
Kahane~\cite{3} (see also~\cite[Sec.~VI. 3]{6}). The proof is
based on the van der Corput lemma and essentially uses
nondegeneration of the curvature of a certain arc of the graph
of $\varphi$.

We note that in general the norms
$\|e^{i\lambda\varphi}\|_{A(\mathbb T)}$ can grow rather
slowly. It was shown by Kahane~\cite{3} (see
also~\cite[Sec.~VI. 2]{6}) that if $\varphi$ is a piecewise
linear nonconstant continuous real function on $\mathbb T$,
then $\|e^{i\lambda\varphi}\|_{A(\mathbb T)}\simeq\log
|\lambda|$. In this respect, let us recall Kahane's conjecture
about the possible essential improvement of the
Beurling--Helson theorem. Namely, Kahane conjectured that the
conclusion of the Beurling--Helson theorem holds under much
weaker assumption that
$\|e^{in\varphi}\|_{A(\mathbb{T})}=o(\log |n|)$. This
conjecture, proposed at the ICM'1962 \cite{4} and later
discussed in~\cite{6,7}, is still unproved. The first
strengthening of the Beurling--Helson theorem in this direction
was obtained by the present author~\cite{14} by means of
methods of number theory and additive combinatorics. Later,
Konyagin and Shkredov~\cite{10} improved the result by
combining the author's approach with a finer technique.

In the multidimensional case it is easy to show that if
$\varphi : \mathbb T^d\rightarrow\mathbb R, ~d\geq 2,$ is of
class $C^s$ with $s>d/2$, then
$\|e^{i\lambda\varphi}\|_{A(\mathbb T^d)}\lesssim
|\lambda|^{d/2}$, see~\cite{2} and \cite[Theorem~3]{13}. The
Leibenson--Kahane's result was extended to the multidimensional
case by Hedstrom~\cite{2}, who showed, that if $\varphi\in
C^2\cap A(\mathbb T^d)$ and the determinant of the matrix of
the second derivatives of $\varphi$ is not identically equal to
zero, then $\|e^{i\lambda\varphi}\|_{A(\mathbb
T^d)}\gtrsim|\lambda|^{d/2}$. This is proved by reduction to
the one-dimensional case.

The estimates for the norms of $e^{i\lambda\varphi}$ for $C^1$
phase functions $\varphi$, including those in the
multidimensional case, were obtained by the author
in~\cite{12,13} \footnote{These papers treat the general case
of the spaces $A_p, 1\leq p<2,$ of functions~$f$ with
$\widehat{f}\in l^p$.}. Certainly, in general, the approach
that uses curvature considerations fails in $C^1$ case. Note
also that for $C^1$ phase functions $\varphi$ the norms of
$e^{i\lambda\varphi}$ can grow nearly as slowly as those for
the piecewise linear functions, namely the author
showed~\cite{11} (see also~\cite{12}) that if $\gamma(\lambda)$
is an arbitrary positive function on $[0, +\infty)$ with
$\gamma(\lambda)\rightarrow\infty$ as
$\lambda\rightarrow+\infty$, then there exists a nowhere linear
(i.e., not linear on any interval) function $\varphi\in
C^1(\mathbb T)$ such that $\|e^{i\lambda\varphi}\|_{A(\mathbb
T)}=O(\gamma(|\lambda|)\log |\lambda|)$.

In the present paper we consider phase functions on $\mathbb
T^d, \,d\geq 2,$ of the following form. Let $a \colon
\mathbb{T}^k\rightarrow\mathbb{R}^m$ and $b \colon
\mathbb{T}^m\rightarrow\mathbb{R}^m$ be two mappings. Define a
function~$\varphi$ on $\mathbb{T}^{k+m}$ as the inner product
\begin{equation}\label{e1}
\varphi(x, y)=(a(x), b(y)), \qquad x\in \mathbb{T}^k,\quad y\in
\mathbb{T}^m.
\end{equation}
We require $a$ and $b$ to be of class $A$ (see the next
section), which guarantees that $e^{i\lambda\varphi}\in
A(\mathbb T^{k+m})$ for all $\lambda\in\mathbb R$. Assuming
that $b$ coincides with a non-degenerate affine mapping in some
domain $J\subseteq[-\pi, \pi]^m$, we obtain lower bounds for
the norms $\|e^{i\lambda\varphi}\|_{A(\mathbb{T}^{k+m})}$. In
particular the class of the phase functions we consider
includes those of the form
\begin{equation}\label{e2}
\varphi(x_1, \ldots, x_k, y_1, \ldots, y_m)=
\sum_{j=1}^m a_j(x_1, \ldots, x_k)|y_j|,
\quad (y_1, \ldots, y_m)\in [-\pi, \pi]^m,
\end{equation}
where $a_j, \,j=1, \ldots, m,$ are real functions in $A(\mathbb
T^k)$. Though these phase functions are of a very special kind,
they, as we will see, provide examples of $\varphi$'s for which
the growth of $\|e^{i\lambda\varphi}\|_{A(\mathbb T^{k+m})}$ is
very fast. Note that we do not assume a smoothness of the
mappings $a$ and $b$ in (1). The fast growth in the case we
consider is a consequence of the geometric structure of
$\varphi$ and not that of its smoothness, namely, the key role
is played by the massiveness of the image of the
torus~$\mathbb{T}^k$ under~$a$. The simplest application of the
results of the present paper pertains to the two-dimensional
case. If $\varphi(x, y)=a(x)|y|$, $|x|\leq \pi$, $|y|\leq\pi$,
where $a\in A(\mathbb{T})$ is an arbitrary nonconstant real
function, then
$\|e^{i\lambda\varphi}\|_{A(\mathbb{T}^2)}\gtrsim |\lambda|$.
When $d=k+m\geq 3$ our results show that the growth that
corresponds to the phase functions of the form (2) can be even
faster then that in the smooth case; for example, if a mapping
$a(x)=(a_1(x), a_2(x))$ of $\mathbb T$ into $\mathbb R^2$ is
\textit{space-filling} (see the next section), then, setting
$\varphi(x, y_1, y_2)=a_1(x)|y_1|+a_2(x)|y_2|$, we have
$\|e^{\lambda\varphi}\|_{A(\mathbb T^3)}\gtrsim |\lambda|^2$.
Note, that, for $\varphi\in C^2(\mathbb T^3)$ we have
$\|e^{i\lambda\varphi}\|_{A(\mathbb T^3)}\lesssim
|\lambda|^{3/2}$, according to the upper bound indicated above.

\section{Statement of the theorem. Corollaries for
space-filling mappings}\label{s2}

We measure the massiveness of a set in terms of its covering
number. Recall that the covering number $N_F(\varepsilon)$ of a
bounded set $F\subset\mathbb{R}^m$ is the smallest number of
balls of radius $\varepsilon>0$ needed to cover~$F$.

Let $s \colon  \mathbb{T}^p\rightarrow\mathbb{R}^q$ be some
mapping. Thus,
$$
 s(x)=(s_1(x), s_2(x), \ldots, s_q(x)), \qquad x\in\mathbb{T}^p.
$$
We say that $s$ is a \textit{mapping of class}~$A$ if all
coordinate functions $s_j : \mathbb T^p\rightarrow\mathbb R$,
$j=1, 2, \ldots, q$, are in~$A(\mathbb{T}^p)$. Clearly, if $a
\colon \mathbb{T}^k\rightarrow\mathbb{R}^m$ and $b \colon
\mathbb{T}^m\rightarrow\mathbb{R}^m$ are mappings of class~$A$,
then the corresponding function~$\varphi$ defined in~\eqref{e1}
belongs to $A(\mathbb{T}^{k+m})$, and hence
$e^{i\lambda\varphi}\in A(\mathbb{T}^{k+m})$ for all
$\lambda\in\mathbb{R}$.

We write $\xi(\lambda)\gtrsim\eta(\lambda)$ if there exists a
constant $c>0$ independent of $\lambda\in\mathbb{R}$ such that
$\xi(\lambda)\geq c\eta(\lambda)$ whenever $|\lambda|$ is
sufficiently large.

\begin{theorem}
Let $a \colon  \mathbb{T}^k\rightarrow \mathbb{R}^m$ and $b
\colon \mathbb{T}^m\rightarrow \mathbb{R}^m$ be mappings of
class~$A$. Assume that $b$ coincides with a non-degenerate
affine mapping in some domain $J\subseteq [-\pi, \pi]^m$. Let
$\varphi$~be the function on $\mathbb{T}^{k+m}$ defined as the
inner product $\varphi(x, y)=(a(x), b(y))$ for $x\in
\mathbb{T}^k$ and $y\in \mathbb{T}^m$. Then
$$
\|e^{i\lambda\varphi}\|_{A(\mathbb{T}^{k+m})}\gtrsim
N_W(1/|\lambda|), \qquad \lambda\in\mathbb{R}, \quad
|\lambda|\rightarrow\infty,
$$
where $W=a(\mathbb{T}^k)$ is the image of the torus $\mathbb{T}^k$
under~$a$.
\end{theorem}

The proof of the theorem is given in Section~\ref{s3}. Mappings
of class~$A$ and some related open problems are discussed in
Remarks at the end of the paper. For now, let us indicate two
immediate corollaries of the theorem.

\begin{corollary}\label{cy1}
Let $a\in A(\mathbb{T})$ and $b\in A(\mathbb{T})$ be real
nonconstant functions. Assume that $b$~is linear on some
interval. Let $\varphi$~be the function on $\mathbb{T}^2$ given
by $\varphi(x, y)=a(x)b(y)$, $x\in\mathbb{T}$,
$y\in\mathbb{T}$. Then
$$
\|e^{i\lambda\varphi}\|_{A(\mathbb{T}^2)}\gtrsim
|\lambda|, \qquad \lambda\in\mathbb{R}, \quad
|\lambda|\rightarrow\infty.
$$
\end{corollary}

In particular, this corollary implies the estimate mentioned at
the end of Introduction: if a function~$\varphi$ on
$\mathbb{T}^2$ has the form $\varphi(x, y)=a(x)|y|$, $-\pi\leq
y\leq\pi$, where $a\in A(\mathbb{T})$ is an arbitrary
nonconstant real function, then
$\|e^{i\lambda\varphi}\|_{A(\mathbb{T}^2)}\gtrsim |\lambda|$.

The second corollary deals with the general case in which the
range of the mapping~$a$ is maximally massive. We say that a
mapping $a \colon \mathbb{T}^k\rightarrow\mathbb{R}^m$ is
\textit{space~-filling} if the image $a(\mathbb{T}^k)$ has
non-empty interior. A~weaker condition is that the Lebesgue
measure of the set $a(\mathbb{T}^k)$ is positive. It is not
difficult to see that there exist space-filling mappings~$a$ of
class~$A$ for any~$k$ and~$m$. Let us verify this. Clearly,
only the case of increasing dimension, that is, the case of
$m>k$, is non-trivial. Note that if $f\in A(\mathbb{T})$, then
the function~$F$ on~$\mathbb{T}^k$ defined by $F(x_1, x_2,
\ldots, x_k)=f(x_1)$ satisfies $F\in A(\mathbb{T}^k)$. Hence it
suffices to consider the case of $k=1$; i.e., it suffices to
show that there exist space-filling curves of class~$A$
in~$\mathbb{R}^m$. Recall that a closed set $E\subseteq
\mathbb{T}$ is called a \textit{Helson set} if every function
continuous on~$E$ is the restriction to~$E$ of some function
in~$A(\mathbb{T})$. Fix a perfect nowhere dense set~$E$ that is
a Helson set. (For the existence of such sets see, e.g.,
~\cite[Theorems~5.2.2 and~5.6.6]{17}.) Let $K$~be a closed cube
in~$\mathbb{R}^m$. Since $E$~is perfect and nowhere dense, it
follows (by the classical Hausdorff--Alexandroff theorem) that
there exists a continuous mapping $\alpha(t)=(\alpha_1(t),
\alpha_2(t), \ldots, \alpha_m(t))$ of~$E$ onto~$K$. Since
$E$~is a Helson set, we see that for each $j=1, 2, \ldots, m$
there exists a function $a_j\in A(\mathbb{T})$ coinciding
with~$\alpha_j$ on~$E$. By taking the real parts, we can assume
that all~$a_j$ are real. Let $a(t)=(a_1(t), a_2(t), \ldots,
a_m(t))$, $t\in \mathbb{T}$. We obtain $a(\mathbb{T})\supseteq
a(E)=\alpha(E)=K$, as desired.

Clearly, if $W\subseteq\mathbb{R}^m$ is an arbitrary set of
positive measure, then
$$
  \inf_{\varepsilon>0}N_W(\varepsilon)\varepsilon^m>0;
$$
hence we obtain the following corollary of the theorem.

\begin{corollary}\label{cy2}
In addition to the assumptions of the theorem, let the mapping
$a \colon \mathbb{T}^k\rightarrow\mathbb{R}^m$ be space-filling
at least in the weak sense\textup; i.e., let $a(\mathbb{T}^k)$
have positive Lebesgue measure. Then
$$
\|e^{i\lambda\varphi}\|_{A(\mathbb{T}^{k+m})}\gtrsim
|\lambda|^m, \qquad \lambda\in\mathbb{R}, \quad
|\lambda|\rightarrow\infty.
$$
\end{corollary}

We note that if $\varphi$ is a sufficiently smooth real
function on~$\mathbb{T}^{k+m}$, then, according to the upper
bound indicated in the introduction, we have
$\|e^{i\lambda\varphi}\|_{A(\mathbb{T}^{k+m})}\lesssim
|\lambda|^{(k+m)/2}$. On the other hand, by taking a
space-filling mapping
$$
(a_1(x), \ldots, a_m(x)) \colon \mathbb{T}^k\rightarrow\mathbb{R}^m
$$
of class $A$ and taking $\varphi$ to be of the form \eqref{e2},
we find in the non-trivial case of dimension-raising~$a$, that
is, for $m>k$, that a substantially faster growth of the $A$
-norms of $e^{i\lambda\varphi}$ occurs. Namely, we have
$\|e^{i\lambda\varphi}\|_{A(\mathbb{T}^{k+m})}
\gtrsim|\lambda|^m$ by Corollary~\ref{cy2}. (Certainly this
effect is possible only when $k+m\geq 3$.)

\section{Proof of the theorem}\label{s3}

Let $A(\mathbb{R}^d)$ be the space of all functions~$f$ of the
form
$$
f(t)=\int_{\mathbb{R}^d}g(\xi)e^{i(x,t)} dx,
$$
where $g\in L^1(\mathbb{R}^d)$. This space is closely related
to~$A(\mathbb{T}^d)$. It is a Banach algebra with respect to
the norm $\|f\|_{A(\mathbb{R}^d)}=\|g\|_{L^1(\mathbb{R}^d)}$
and the usual multiplication of functions. It is convenient to
give an equivalent definition by saying that $A(\mathbb{R}^d)$
contains every continuous bounded function on~$\mathbb{R}^d$
whose Fourier transform~$\widehat{f}$ in the sense of tempered
distributions belongs to~$L^1(\mathbb{R}^d)$. Naturally we let
$\|f\|_{A(\mathbb{R}^d)}=\|\widehat{f}\|_{L^1(\mathbb{R}^d)}$.
Here the normalizing factor in the Fourier transform is chosen
so that
$$
\widehat{f}(x)=\frac{1}{(2\pi)^d}\int_{\mathbb{R}^d}f(t)e^{-i(x,t)}
dt
$$
for $f\in L^1(\mathbb{R}^d)$. Technically, it is more
convenient to use the $A(\mathbb{R}^d)$ -norm, so the symbol
$\,\widehat{}\,$ will denote the Fourier transform
in~$\mathbb{R}^d$ everywhere in the proof of the theorem.

We will use the following notation. If $\xi\in\mathbb{R}^d$,
then $|\xi|$~stands for the length of the vector~$\xi$. Let
$F\subseteq\mathbb{R}^d$ be an arbitrary set.
By~$(F)_\varepsilon$ we denote the
$\varepsilon$~\nobreakdash-neighbourhood of~$F$; i.e.,
$(F)_\varepsilon=\{t\in\mathbb{R}^d \colon  \inf_{x\in
F}|t-x|\leq\varepsilon\}$. For $\lambda\in\mathbb{R}$, let
$\lambda F=\{\lambda x, x\in F\}$. If $F$~is measurable, then
$|F|$~stands for its (Lebesgue) measure. The symbol~$\ast$
denotes the convolution of functions in~$L^1(\mathbb{R}^d)$ or
the convolution of measures. This notation is used for
various~$d$, but this will not lead to a misunderstanding.

The proof of the theorem is based on a modification of our
method, that we used in ~\cite{12,13} for $C^1$ -phase
functions (it can be called the concentration of Fourier
transform large values method).

We will need some preliminary constructions and lemmas.

For $\delta>0$, let $\Delta_\delta$~be the ``triangle
function'' supported by the interval $(-\delta, \delta)$; that
is,
$$
\Delta_\delta(t)=\max(0, 1-|t|/\delta), \qquad t\in\mathbb{R}.
$$
It is well known (and easy to verify) that
\begin{equation}\label{e3}
\widehat{\Delta_\delta}(u)=\frac{2}{\pi}\frac{\sin^2 (\delta
u/2)}{\delta u^2}, \quad u\in\mathbb{R}\setminus\{0\}; \qquad
\widehat{\Delta_\delta}(0)=\frac{\delta}{2\pi}.
\end{equation}
Note that since $\widehat{\Delta_\delta}(u)\geq 0$ for all
$u\in\mathbb{R}$, it follows that
$$
\|\Delta_\delta\|_{A(\mathbb{R})}=\Delta(0)=1.
$$

For an arbitrary interval $Q\subseteq\mathbb{R}$, let
$\Delta_Q$ be the triangle function supported by~$Q$; that is,
$\Delta_Q(t)=\Delta_{|Q|/2}(t-c_Q)$, where~$c_Q$ is the center
of the interval~$Q$. One has
$|\widehat{\Delta_Q}|=\widehat{\Delta_{|Q|/2}}$ and hence
\begin{equation}\label{e4}
\|\Delta_Q\|_{A(\mathbb{R})}=1.
\end{equation}

Let $Q_1, Q_2, \ldots Q_d$ be intervals in~$\mathbb{R}$. For
the parallelepiped $Q=Q_1\times Q_2\times\ldots \times
Q_d\subseteq\mathbb{R}^d$, set
\begin{equation}\label{e5}
\Delta_Q(t)=\Delta_{Q_1}(t_1)\Delta_{Q_2}(t_2)\ldots\Delta_{Q_d}(t_d),
\qquad t=(t_1, t_2, \ldots, t_d)\in\mathbb{R}^d.
\end{equation}
Obviously (see~\eqref{e4}),
\begin{equation}\label{e6}
\|\Delta_Q\|_{A(\mathbb{R}^d)}=1.
\end{equation}
In addition, if~$Q$ is centered at the origin, then
$\widehat{\Delta_Q}\geq 0$.

The following simple lemma is of a technical nature.

\begin{lemma}\label{l1}
Let $Q\subseteq\mathbb{R}^d$ be a cube with edges of length
$2\delta$ parallel to the coordinate axes. Then
$|\widehat{\Delta_Q}(u)|\geq 4^{-d}(\delta/2\pi)^d$ for all
$u\in (-1/\delta, 1/\delta)^d$.
\end{lemma}

\begin{proof}
Since $|\sin\alpha|\geq |\alpha|/2$ for $|\alpha|\leq 1$,
from~\eqref{e3} it follows that
$$
\widehat{\Delta_\delta}(u)\geq\frac{\delta}{8\pi} \quad
\textrm{for} \quad |u|\leq 1/\delta.
$$
Since the cube $Q$ is obtained by a shift of the cube
$(-\delta, \delta)^d$, we have
$|\widehat{\Delta_Q}(u)|=\widehat{\Delta_{(-\delta,
\delta)^d}}(u)$. It remains to use relation~\eqref{e5}. The
proof of the lemma is complete.
\end{proof}

Let $\omega$ be the modulus of continuity of the mapping~$a$; i.e.,
$$
\omega(\delta)=\sup_{|x_1-x_2|\leq \delta}|a(x_1)-a(x_2)|,
\quad \delta\geq 0.
$$
The function $\omega(\delta)$ is non-decreasing and continuous
on $[0, +\infty)$, and $\omega(0)=0$. We will assume that the
mapping~$a$ is nonconstant; otherwise the assertion of the
theorem is obvious. Thus, $\omega(\delta)>0$ for
all~$\delta>0$.

By the assumption of the theorem, $b(y)=Py+y_0$ in some domain
$J\subseteq [-\pi, \pi]^m$, where $y_0\in\mathbb{R}^m$ and $P$
is a real $m\times m$ matrix with $\det P\neq 0$. Without loss
of generality, we can assume that $J=(-l, l)^m$, where
$0<l<\pi$. By~$P^\ast$ we denote the transpose of~$P$.

Set
$$
\rho=\sup_{y\in J}|Py+y_0|.
$$
Fix a constant $\varepsilon_0$ such that
\begin{equation}\label{e7}
0<\varepsilon_0\leq 1/2, \qquad 2\rho
\varepsilon_0\leq\frac{1}{2}4^{-k}.
\end{equation}

For each sufficiently large $\lambda>0$, define $\delta_\lambda$ by
the condition
\begin{equation}\label{e8}
\lambda\omega(\sqrt{k}2\delta_\lambda)=\varepsilon_0.
\end{equation}

The following lemma is the key claim for the proof of the
theorem.

\begin{lemma}\label{l2}
Let $\lambda>0$ be sufficiently large. Then for each $v\in (\lambda
W)_{\varepsilon_0}$ there exists a cube $I_{\lambda, v}\subseteq
[-\pi, \pi]^k$ with edges of length $2\delta_\lambda$ parallel to
coordinate axes such that
$$
|(\Delta_{I_{\lambda, v}\times J}e^{i\lambda\varphi})^\wedge(u,
P^\ast v)|\geq c\delta_\lambda^k
$$
for all $u\in (-1/\delta_\lambda, 1/\delta_\lambda)^k$. The
constant $c>0$ is independent of $u$, $v$, and~$\lambda$.
\end{lemma}

\begin{proof}
Let $v\in (\lambda W)_{\varepsilon_0}$. Then one can find a
point $x_{\lambda, v}\in [-\pi, \pi]^k$ such that
\begin{equation}\label{e9}
|v-\lambda a(x_{\lambda, v})|\leq \varepsilon_0.
\end{equation}
Let us assume that $\lambda>0$ is so large that
$2\delta_\lambda<2\pi$. Then we can find a (closed) cube
$I_{\lambda, v}\subseteq [-\pi, \pi]^k$, with edges of length
$2\delta_\lambda$ parallel to coordinate axes, such that
$x_{\lambda, v}\in I_{\lambda, v}$. Let us verify that the
conclusion of the lemma holds for this cube.

Define a function $\varphi_{\lambda, v}$ by setting
$$
\varphi_{\lambda, v}(x, y)=\bigg(\frac{1}{\lambda}v, Py+y_0\bigg),
\qquad x\in \mathbb{R}^k, \quad y\in\mathbb{R}^m.
$$
Let $x\in I_{\lambda, v}$ and $y\in J$. Then (see~\eqref{e8},
\eqref{e9})
\begin{align*}
 |\varphi(x, y)-\varphi_{\lambda, v}(x, y)|&=|(a(x),
Py+y_0)-(\lambda^{-1}v, Py+y_0)|
\\
&=|(a(x)-\lambda^{-1}v, Py+y_0)|\leq\rho|a(x)-\lambda^{-1}v|
\\
&\leq \rho|a(x)-a(x_{\lambda, v})|+\rho|a(x_{\lambda,
v})-\lambda^{-1}v|
\\
&\leq\rho\omega(\sqrt{k}2\delta_\lambda)
+\rho\varepsilon_0/\lambda=2\rho\varepsilon_0/\lambda.
\end{align*}
So,
$$
|e^{i\lambda\varphi(x, y)}-e^{i\lambda\varphi_{\lambda, v}(x, y)}|\leq 2\rho\varepsilon_0.
$$
Hence for all $u\in\mathbb{R}^k$, in view of~\eqref{e7}, we
obtain
\begin{equation}\label{e10}
\begin{split}
|(\Delta_{I_{\lambda, v}\times J}&e^{i\lambda\varphi})^\wedge(u,
P^\ast v)- (\Delta_{I_{\lambda, v}\times J}
e^{i\lambda\varphi_{\lambda, v}})^\wedge(u, P^\ast v)|
\\
&\leq\frac{1}{(2\pi)^{k+m}}\iint_{\mathbb{R}^k\times\mathbb{R}^m}
\Delta_{I_{\lambda, v}\times J}(x, y) |e^{i\lambda\varphi(x,
y)}-e^{i\lambda\varphi_{\lambda, v}(x, y)}|dx dy
\\
&\leq|\widehat{\Delta_{I_{\lambda, v}\times J}}(0)|2\rho
\varepsilon_0= 2\rho
\varepsilon_0|\widehat{\Delta_{J}}(0)||\widehat{\Delta_{I_{\lambda,
v}}}(0)|
\\
&\leq\frac{1}{2}4^{-k}|\widehat{\Delta_{J}}(0)|\bigg(\frac{\delta_\lambda}{2\pi}\bigg)^k.
\end{split}
\end{equation}
At the same time,
\begin{align*}
(\Delta_{I_{\lambda, v}\times J}& e^{i\lambda\varphi_{\lambda,
v}})^\wedge(u, P^\ast v)
\\
&=\frac{1}{(2\pi)^{k+m}} \int_{x\in I_{\lambda, v}}\int_{y\in J}
\Delta_{I_{\lambda, v}}(x) \Delta_J(y)e^{i(v, Py+y_0)}e^{-i(u,
x)}e^{-i(P^\ast v, y)} dx dy
\\
&=e^{i(v, y_0)}\widehat{\Delta_J}(0) \widehat{\Delta_{I_{\lambda,
v}}}(u),
\end{align*}
and hence it follows from Lemma~\ref{l1} that
$$
|(\Delta_{I_{\lambda, v}\times J} e^{i\lambda\varphi_{\lambda,
v}})^\wedge(u, P^\ast v)|\geq |\widehat{\Delta_J}(0)|4^{-k}
\bigg(\frac{\delta_\lambda}{2\pi}\bigg)^k
$$
for $u\in (-1/\delta_\lambda, 1/\delta_\lambda)^k$. Taking
\eqref{e10} into account, we complete the proof of the lemma.
\end{proof}

Let us proceed directly to the proof of the theorem. Without
loss of generality, we can assume that $\lambda>0$ is
sufficiently large. Recall that $J=(-l, l)^m$, where $0<l<\pi$.

Consider the expansion
$$
e^{i\lambda\varphi(t)}=\sum_{n\in\mathbb{Z}^{k+m}} c_\lambda(n)
e^{i(n, t)}, \qquad t\in\mathbb{T}^{k+m},
$$
and define measures $\mu_\lambda$ and $\sigma_\lambda$ by setting
$$
\mu_\lambda=\sum_{n\in\mathbb{Z}^{k+m}}c_\lambda(n)\delta_n, \qquad
\sigma_\lambda=\sum_{n\in\mathbb{Z}^{k+m}}|c_\lambda(n)|\delta_n,
$$
where $\delta_{t}$ stands for the unit mass concentrated at a
point $t\in\mathbb{R}^{k+m}$. Let $I_{\lambda, v}$ be the cube
whose existence has been established in Lemma~\ref{l2}. Since
the parallelepiped $I_{\lambda, v}\times J$ is obtained by a
shift of the parallelepiped $(-\delta_\lambda,
\delta_\lambda)^k\times (-l, l)^m$, we have
$$
|(\Delta_{I_{\lambda, v}\times
J})^\wedge|=(\Delta_{(-\delta_\lambda,
\delta_\lambda)^k\times (-l, l)^m})^\wedge.
$$
Therefore,
\begin{align*}
|(\Delta_{I_{\lambda, v}\times J}&e^{i\lambda\varphi})^\wedge(u,
P^\ast v)| =|(\Delta_{I_{\lambda, v}\times
J})^\wedge\ast\mu_\lambda(u, P^\ast v)|
\\
&\leq|(\Delta_{I_{\lambda, v}\times
J})^\wedge|\ast\sigma_\lambda(u, P^\ast v)=
(\Delta_{(-\delta_\lambda, \delta_\lambda)^k\times (-l,
l)^m})^\wedge\ast\sigma_\lambda(u, P^\ast v)
\end{align*}
for all $u\in\mathbb{R}^k$ and $v\in\mathbb{R}^m$. Thus, using
Lemma~\ref{l2}, we see that the estimate
$$
c\delta_\lambda^k\leq (\Delta_{(-\delta_\lambda,
\delta_\lambda)^k\times (-l, l)^m})^\wedge\ast\sigma_\lambda(u,
P^\ast v)
$$
holds for all
$$
(u, v)\in \bigg(-\frac{1}{\delta_\lambda},
\frac{1}{\delta_\lambda}\bigg)^k\times (\lambda W)_{\varepsilon_0}.
$$
Hence (see~\eqref{e6})
\begin{equation}\label{e11}
\begin{split}
c|(\lambda W)_{\varepsilon_0}|&\leq
c\delta_\lambda^k\bigg|\bigg(-\frac{1}{\delta_\lambda},
\frac{1}{\delta_\lambda}\bigg)^k \times(\lambda
W)_{\varepsilon_0}\bigg|
\\
&\leq\iint_{\big(-\frac{1}{\delta_\lambda},
\frac{1}{\delta_\lambda}\big)^k \times(\lambda
W)_{\varepsilon_0}}|(\Delta_{(-\delta_\lambda,
\delta_\lambda)^k\times (-l, l)^m})^\wedge\ast\sigma_\lambda(u,
P^\ast v)|du dv
\\
&\leq\iint_{\mathbb{R}^k\times\mathbb{R}^m}|(\Delta_{(-\delta_\lambda,
\delta_\lambda)^k\times (-l, l)^m})^\wedge\ast\sigma_\lambda(u,
P^\ast v)|du dv
\\
&=\frac{1}{|\det
P|}\iint_{\mathbb{R}^k\times\mathbb{R}^m}(\Delta_{(-\delta_\lambda,
\delta_\lambda)^k\times (-l, l)^m})^\wedge\ast\sigma_\lambda(u,
v)du dv
\\
&=\frac{1}{|\det P|}\sum_n |c_\lambda(n)|= \frac{1}{|\det
P|}\|e^{i\lambda\varphi}\|_{A(\mathbb{T}^{k+m})}.
\end{split}
\end{equation}
At the same time, it is well known (e.g., see~\cite[Secs.~5.4
and~5.6]{16}) that
$$
v_d N_F(2\varepsilon)\varepsilon^d\leq |(F)_\varepsilon|
$$
for an arbitrary bounded set $F\subseteq\mathbb{R}^d$,
where~$v_d$ is the volume of the unit ball in~$\mathbb{R}^d$.
Thus, since $\varepsilon_0\leq 1/2$ (see~\eqref{e7}), for the
left-hand side in~\eqref{e11} we obtain
$$
|(\lambda W)_{\varepsilon_0}|\gtrsim N_{\lambda
W}(2\varepsilon_0)\geq N_{\lambda W}(1)=N_W(1/\lambda).
$$
The proof of the theorem is complete.

\quad

\textsc{Remarks.} 1. Recall that, as we have verified in
Section~2, for each~$m$ there exists a space-filling curve of
class~$A$ in~$\mathbb{R}^m$. Actually some explicit examples of
space-filling curves of class~$A$ in~$\mathbb{R}^2$
and~$\mathbb{R}^3$ are well known. For example, the
corresponding mappings can be obtained as sums of lacunary
power series. The first results of this kind go back to Salem
and Zygmund. Their approach was later developed
in~\cite[Theorem~II]{8}, where it was shown that if $\sigma>1$
and $\inf_k n_{k+1}/n_k>1$, then the mapping
$$
a(t)=\sum_{k=1}^\infty \frac{1}{k^\sigma}e^{in_kt}
$$
is space-filling in the complex plane. Another example
in~$\mathbb{R}^2$ is the Shoenberg curve
(see~\cite[Chap.~7]{18}); it is easily seen that the Shoenberg
curve is of class~$A$, the same applies to the Shoenberg curve
in~$\mathbb{R}^3$. As to the other classical space-filling
curves that can be found in~\cite{18} the author does not know
if, being properly modified to obtain continuous mappings of
the circle, they are of class~$A$. In particular, consider the
classical Peano mapping $p \colon [0,
1]\rightarrow\mathbb{R}^2$. Define a mapping $\widetilde{p} :
\mathbb T\rightarrow\mathbb R^2$ by setting
$\widetilde{p}(t)=p(|t|/\pi)$ for $t\in [-\pi, \pi]$. Is it
true that $\widetilde{p}$ is of class~$A$?

2. It is well known that the covering number of a set is
closely related to its dimension (see, e.g., \cite[Chaps.
4--5]{16}). Assuming that Hausdorff dimension $\dim_H W$ or
Minkowski dimension $\dim_M W$ of a set $W$ equals $s_0$, we
have $N_W(\varepsilon)\gtrsim 1/\varepsilon^{s}$ for every $s$
with $s<s_0$. Thus, if $a : \mathbb T^k\rightarrow\mathbb R^m$
is a mapping of class $A$ such that the image $W=a(\mathbb
T^k)$ is of Hausdorff or Minkowski dimension $s_0$, then for
the corresponding phase function $\varphi$ the theorem of the
present work yields $\|e^{i\lambda\varphi}\|_{A(\mathbb
T^{k+m})}\gtrsim |\lambda|^{s}$ for every $s$ with $s<s_0$. The
fact that for each $s_0$ with $1\leq s_0\leq m$, there exists a
mapping $a \colon \mathbb{T}\rightarrow\mathbb{R}^m$ of
class~$A$ such that the image $W=a(\mathbb{T})$ has Hausdorff
dimension~$s_0$ readily follows from the well known theorem on
the range of the sum of a Gaussian trigonometric
series~\cite[Sec.~XIV.4, Th. 1]{5}. Namely, let $X_n$~and~$
Y_n$ be similar independent Gaussian random variables in
$\mathbb{R}^m$. Consider the Gaussian trigonometric series of
the form
$$
a(t)=\sum_{n=0}^\infty 2^{-n/s_0} (X_n\cos 2^nt+Y_n\sin 2^nt).
$$
We have $\dim_H a(\mathbb{T})=s_0$ almost surely. We also note
that explicit examples of planar curves of class $A$ of a given
dimension can be obtained by using the Weierstrass function
$$
w(t)=\sum_{n\geq 0}\frac{1}{2^{(2-s_0)n}}\cos 2^nt,
$$
where $1<s_0<2$. It was proved in~\cite{9} that the Minkowski
dimension of the graph of~$w$ on $[-\pi, \pi]$ is~$s_0$. By
setting $a(t)=(|t|, w(t))$ for $t\in [-\pi, \pi]$, we obtain a
mapping $a \colon \mathbb{T}\rightarrow\mathbb{R}^2$ of
class~$A$ such that $\dim_M a(\mathbb{T})=s_0$.

3. In general it is not clear what sets in $\mathbb R^m,
\,m\geq 2,$ can be obtained as images of mappings $a : \mathbb
T^k\rightarrow\mathbb R^m$ of class~$A$ (for $m>k$). Let $W$ be
a compact set in~$\mathbb{R}^m$. Assume that $W$ is a
continuous image of the circle~$\mathbb{T}$. Is it true that
$W$~is the image of~$\mathbb{T}$ under some mapping of
class~$A$?  The author does not know the answer even for $m=2$.

\quad

\textbf{Acknowledgements.} The article was prepared within the
framework of the Academic Fund Program at the National Research
University Higher School of Economics (HSE) in 2016--2017
(grant no.~16-01-0078) and supported within the framework of a
subsidy granted to the HSE by the Government of the Russian
Federation for the implementation of the Global Competitiveness
Program.

\end{document}